\documentclass{amsart}
\usepackage{amsmath,amsfonts,amsthm}
\usepackage[english]{babel}

\newtheorem{theorem}{Theorem}[section]
\newtheorem{lemma}[theorem]{Lemma}

\newtheorem{proposition}[theorem]{Proposition}
\newtheorem{corollary}[theorem]{Corollary}

\newtheorem{question}[theorem]{Question}

\numberwithin{equation}{section}

\begin{document}

\title{Diagonals of separately continuous functions and their analogs}

\author{Olena Karlova}

\author{Volodymyr Mykhaylyuk}

\author{Oleksandr Sobchuk}

\begin{abstract}
  We prove that for a topological space $X$, an equiconnected space $Z$ and a Baire-one mapping $g:X\to Z$ there exists a separately continuous mapping  $f:X^2\to Z$ with the diagonal $g$, i.e. $g(x)=f(x,x)$ for every $x\in X$. Under a mild assumptions on $X$ and  $Z$ we obtain that diagonals of separately continuous mappings $f:X^2\to Z$ are exactly Baire-one functions, and diagonals of mappings $f:X^2\to Z$ which are continuous on the first variable and Lipschitz (differentiable) on the second one, are exactly the functions of stable first Baire class.
\end{abstract}


\maketitle

\section{Introduction}

Let $f:X^2\to Y$ be a mapping. We call a mapping  $g:X\to Y$, $g(x)=f(x,x)$, {\it the diagonal of $f$.}

Investigations of diagonals of separately continuous functions $f:X^2\to\mathbb R$ started in the classical work of R.~Baire \cite{Baire}, who showed that diagonals of separately continuous functions of two real variables are exactly Baire-one functions, i.e. pointwise limits of continuous functions. His result was generalized by A.~Lebesgue and H.~Hahn for real-valued functions of $n$ real variables (see \cite{Leb1, Leb2, Ha}).

Since the second half the XX-th century, Baire classification of separately continuous mappings {and their} analogs {was} intensively studied by many mathe\-ma\-ticians (see \cite{Mo,Ru,Ve,MMMS,B,Bur}). Let us observe that W.~Rudin \cite{Ru} was the first who applied partitions of unity to classify separately continuous functions defined on {the} product of a metrizable {space} and a topological {space} and with values in locally convex spaces. Later this method has been developed and now a partition of unity is a familiar tool for Baire classification of separately continuous functions.

A problem {on construction} of separately continuous functions of $n$ variables with the given diagonal {of $(n-1)$-th} Baire class was {solved} in  \cite{MMS1}. In \cite{My} it was shown that for any topological space $X$ and a function $g:X\to\mathbb R$ {of $(n-1)$-th} Baire class there exist a separately continuous function  $f:X^n\to\mathbb R$ with the diagonal $g$.

Also naturally appear problems on diagonals of $CL$- and $CD$-mappings, i.e. diagonals of mappings  $f:X^2\to Y$ which are continuous with respect to the first variable and Lipschitz or differentiable with respect to the second variable.

In the given paper we generalize the result from~\cite{My} which concerns a construction of separately continuous mappings $f:X^2\to Z$ with {a} given diagonal of the first Baire class {to} the case of an equiconnected space $Z$. We also describe diagonals of $CL$- and $CD$-mappings {for wide classes} of spaces $X$ and $Z$.

\section{Preliminaries}

By $P(X,Y)$ we will denote the collection of all mappings \mbox{$f:X\to Y$} with a property $P$. In particular, by $C(X,Y)$ we denote the collection of all continuous mappings between topological spaces  $X$ and  $Y$.

For a mapping
\mbox{$f:X\times Y\to Z$} and a point $(x,y)\in X\times Y$ we write $$f^x(y)=f_y(x)=f(x,y).$$

Let $P$ and $Q$ be some properties of mappings. We will use the symbol $PQ(X\times Y,Z)$ to denote the collection of all mappings $f:X\times Y\to Z$ which have a property $P$ with respect to the first variable and a property $Q$ with respect to the second one, i.e. $\{y\in Y: f_y\in P(X,Z)\}=Y$  and $\{x\in X: f^x\in Q(Y,Z)\}=X$. If $f\in PP(X\times Y,Z)$, then $f$ is called {\it a mapping which separately has a property $P$}. So, for topological spaces $X$, $Y$ and $Z$ mappings from the class $CC(X\times Y,Z)$ are said to be {\it separately continuous}.

Recall that a mapping $f$ between topological spaces $X$ and $Y$ is called {\it a mapping of the first Baire class} or {\it a Baire-one mapping} if there exists a sequence $(f_n)^{\infty}_{n=1}$ of continuous mappings $f_n:X\to Y$ which is pointwise convergent to $f$, i.e. $\lim\limits_{n\to\infty}f_n(x)=f(x)$ for every $x\in X$. The collection of all Baire-one mappings $f:X\to Y$ is denoted by $B_1(X,Y)$.

{\it An equiconnected space} is a pair $(X,\lambda)$, consisting of a topological space $X$ and a continuous function $\lambda:X\times X\times [0,1]\to X$ such that

{ (i)} $\lambda(x,y,0)=x$;

{ (ii)} $\lambda(x,y,1)=y$;

{ (iii)} $\lambda(x,x,t)=x$\\ for all $x,y\in X$ and $t\in [0,1]$.

\noindent The simplest example of {an} equiconnected space is a convex subset of a topological vector space {with $\lambda$ defined by}  $\lambda(x,y,t)=(1-t)x+t y$.

For a metric space $X$ by $|\cdot - \cdot|_X$ we denote {the} metric on this space.

\section{A method of constructing of separately continuous mappings}

In the following theorem we describe a method of a construction of separately continuous mappings with a given diagonal.

\begin{theorem}\label{th:3.1}
  Let $X$ be a topological space, $(Z,\lambda)$ a metric equiconnected space, let $(G_n)_{n=0}^{\infty}$ and $(F_n)_{n=0}^{\infty}$ be sequences of open sets $G_n$ and closed sets $F_n$ in $X^2$, $(\varphi_n:X^2\to [0,1])_{n=1}^{\infty}$ be a sequence of separately continuous functions,  $(g_n:X\to Z)^{\infty}_{n=1}$ be a sequence of continuous functions pointwise convergent to a function $g:X\to Z$, and $(\delta_n:X\to (0,+\infty))^{\infty}_{n=1}$ be a sequence of positive functions pointwise convergent to zero such that:

     {\rm (1)} $G_0=F_0=X^2$ and $\Delta=\{(x,x):x\in X\}\subseteq G_{n+1}\subseteq F_n\subseteq G_n$ for every $n\in\mathbb N$;

      {\rm (2)}  $X^2\setminus G_n\subseteq\varphi_n^{-1}(0)$ and $F_n\subseteq \varphi_n^{-1}(1)$ for every $n\in\mathbb N$;

       {\rm (3)} $|g_{n}(x)-g_{n}(y)|_Z\leq \delta_n(y)$ and $|g_{n+1}(x)-g_{n+1}(y)|_Z\leq \delta_n(y)$ for every $n\in\mathbb N$ and $(x,y)\in G_n$.

 \noindent Then the mapping $f:X^2\to Z$,
 \begin{equation}\label{eq:3.0}
 f(x,y)=\left\{\begin{array}{ll}
                         \lambda(g_n(x),g_{n+1}(x),\varphi_n(x,y)), & (x,y)\in F_{n-1}\setminus F_n\\
                         g(x), & (x,y)\in E=\bigcap\limits_{n=1}^{\infty} G_n.
                       \end{array}
 \right.
 \end{equation}
is separately continuous.
\end{theorem}

\begin{proof}
{We fix} $n\in\mathbb N$ and show that
\begin{equation}\label{eq:3.1}
f(x,y)=\lambda(\lambda(g_n(x),g_{n+1}(x),\varphi_n(x,y)), g_{n+2}(x),\varphi_{n+1}(x,y))
\end{equation}
for all $(x,y)\in F_{n-1}\setminus F_{n+1}$.

Let $(x,y)\in F_{n-1}\setminus F_{n}$. Since $G_{n+1}\subseteq F_n$, $(x,y)\not\in G_{n+1}$ and $(2)$ imply that $\varphi_{n+1}(x,y)=0$. Therefore, $$\lambda(\lambda(g_n(x),g_{n+1}(x),\varphi_n(x,y)), g_{n+2}(x),\varphi_{n+1}(x,y))=
$$
$$
=\lambda(g_n(x),g_{n+1}(x),\varphi_n(x,y))=f(x,y).$$

Now let $(x,y)\in F_{n}\setminus F_{n+1}$. According to~$(2)$, $\varphi_{n}(x,y)=1$ and $$\lambda(\lambda(g_n(x),g_{n+1}(x),\varphi_n(x,y)), g_{n+2}(x),\varphi_{n+1}(x,y))=$$$$=\lambda(g_{n+1}(x),g_{n+2}(x),\varphi_{n+1}(x,y))=f(x,y).$$

Since the mappings  $\lambda$, $g_n$, $g_{n+1}$, $g_{n+2}$ are continuous, $\varphi_n$ and $\varphi_{n+1}$ are separately continuous, $f$ is  continuous with respect to every variable (separately continuous) at each point of the open set  $G_n\setminus F_{n+1}$ for every $n\in\mathbb N$ as a composition of continuous mappings. Moreover, $f$ is separately continuous at every point of the open set $G_0\setminus F_1=F_0\setminus F_1$. Therefore, $f$ is separately continuous on the open set  $X^2\setminus E=\bigcup\limits_{n=1}^{\infty}(G_{n-1}\setminus F_n)$.

It remains to check the continuity of $f$ with respect to  $x$ and $y$ at the points of $E$. We first notice that $(x,y)\in G_n$ for any point $(x,y)\in E$, hence, in accordance with~$(4)$, $|g_{n}(x)-g_{n}(y)|_Z\leq\delta_n(y)$ for every $n\in\mathbb N$. After {letting} $n\to\infty$, we conclude $g(y)=g(x)$.

Fix $(x_0,y_0)\in E$, $\varepsilon>0$ and denote $z_0=f(x_0,y_0)=g(x_0)=g(y_0)$. Since  $\lambda$ is continuous and $\lambda(z_0,z_0,\alpha)=z_0$ for every $\alpha\in[0,1]$, there exists a neighborhood $W=\{z\in Z:|z-z_0|_Z<\delta\}$ of $z_0$ such that $|\lambda(z_1,z_2,\alpha)-z_0|_Z<\varepsilon$ for any $z_1,z_2\in W$ and $\alpha\in[0,1]$. Since $\lim\limits_{n\to\infty}g_n(x_0)=\lim\limits_{n\to\infty}g_n(y_0)=z_0$ and $\lim\limits_{n\to\infty}\delta_n(y_0)=0$, there exists a number $n_0\in\mathbb N$ such that $g_n(x_0)\in W$, $|g_n(y_0)-z_0|_Z<\frac{\delta}{2}$ and $\delta_n(y_0)<\frac{\delta}{2}$ for every $n\geq n_0$. {Then we choose} neighborhoods $U_0$ and $V_0$ of the points $x_0$ and $y_0$, respectively, such that  $A=(U_0\times\{y_0\})\cup(\{x_0\}\times V_0)\subseteq G_{n_0}$ and show that $|f(x,y)-z_0|_Z<\varepsilon$ for every $(x,y)\in A$.

 Let $(x,y)\in A\cap E$. If $y=y_0$, then $f(x,y_0)=g(x)=g(y_0)=z_0$, and if  $x=x_0$, then $f(x,y_0)=g(x_0)=z_0$. Hence, $f(x,y)=z_0$ and $|f(x,y)-z_0|_Z<\varepsilon$.

Let $(x,y_0)\in A\setminus E$. {We choose} a number $n\in\mathbb N$ such that $(x,y_0)\in G_{n}\setminus G_{n+1}$. Since $(x,y_0)\in G_{n_0}$, $n\geq n_0$. According to (\ref{eq:3.1}), $$f(x,y_0)=\lambda(g_n(x),g_{n+1}(x),\varphi_n(x,y_0)).$$ Moreover, by $(3)$ and by choosing of $n_0$,
$|g_{n}(x)-g_{n}(y_0)|_Z\leq \delta_n(y_0)<\frac{\delta}{2}$ and $|g_{n+1}(x)-g_{n+1}(y_0)|_Z\leq \delta_n(y_0)<\frac{\delta}{2}$. Taking into account that $|g_n(y_0)-z_0|_Z<\frac{\delta}{2}$ and $|g_{n+1}(y_0)-z_0|_Z<\frac{\delta}{2}$, we deduce that $g_{n}(x),g_{n+1}(x)\in W$. Therefore, $|f(x,y_0)-z_0|_Z<\varepsilon$.

Now let $(x_0,y)\in A\setminus E$, $n\geq n_0$ and $(x_0,y)\in G_{n}\setminus G_{n+1}$. Then $f(x_0,y)=\lambda(g_n(x_0),g_{n+1}(x_0),\varphi_n(x_0,y))$, $g_{n}(x_0),g_{n+1}(x_0)\in W$ and $|f(x_0,y)-z_0|_Z<\varepsilon$.

Hence, $f$ is continuous with respect to the variables $x$ and $y$ at the points of the set $E$ and the proof is complete.
\end{proof}

Recall that a subset $A$ of a topological space $X$ is called {\it a zero (co-zero) set} if there exists a continuous function $f:X\to [0,1]$ such that $A=f^{-1}(0)$ ($A=f^{-1}((0,1])$). Zero and co-zero sets are alternatively called {\it functionally closed} and {\it functionally open sets}, respectively.

\begin{corollary}\label{cor:3.2}
 Let $X$ be a topological space, $(Z,\lambda)$ a metrizable equiconnected space and $g\in B_1(X,Z)$. Then there exists a separately continuous mapping $f:X^2\to Z$ with the diagonal $g$.
\end{corollary}

\begin{proof}
Fix any metric $|\cdot - \cdot|_X$ on $X$ which generates its topological structure. Choose a sequence $(g_n)^{\infty}_{n=1}$ of continuous mappings $g_n:X\to Z$ which is pointwise convergent to $g$. Set $G_0=F_0=X^2$, $$G_n=\{(x,y)\in X^2: |g_k(x)-g_k(y)|_Z<\frac1n\,\,\,\,\,\forall\,k\leq n+1\}$$ and $$F_n=\{(x,y)\in X^2: |g_k(x)-g_k(y)|_Z\leq\frac{1}{n+1}\,\,\,\,\,\forall\,k\leq n+1\}.$$ Let us observe that the sequences $(g_n)^{\infty}_{n=1}$, $(G_n)_{n=0}^{\infty}$ and $(F_n)_{n=0}^{\infty}$ satisfy conditions $(1)$ and $(3)$ of Theorem~\ref{th:3.1}. Moreover, all the sets $G_n$ and $F_n$ are co-zero and zero sets, respectively. Then for every $n\in\mathbb N$ there exists a continuous function $\varphi_n:X^2\to[0,1]$ with $\varphi_n^{-1}(0)=X^2\setminus G_n$ and $\varphi_n^{-1}(1)=F_n$. Clearly, the sequence $(\varphi_n)_{n=1}^{\infty}$ satisfies~$(2)$. Now Theorem~\ref{th:3.1} implies the existence of a separately continuous mapping $f$ with {the} diagonal $g$.
\end{proof}

The {following} result gives a characterization of diagonals of separately continuous mappings.
\begin{theorem}\label{cor:3.3}
 Let $X$ be a topological space, $(Z,\lambda)$ a metrizable equiconnected space, $g:X\to Z$ a mapping and let, {at least} one of the following conditions {hold}:

  $(1)$ \,\,$CC(X\times X,Z)\subseteq B_1(X^2,Z)$;

  $(2)$\,\,$X$ is a metrically quarter-stratifiable strongly countably-dimensional paracompact space;

  $(3)$\,\, $X$ is a metrically quarter-stratifiable space and  $Z$ is locally convex;

  $(4)$\,\, $X$ is a metrizable space.

  \noindent Then the following conditions are equivalent:

 $(i)$\,\,\,$g\in B_1(X,Z)$;

 $(ii)$\,\,\,there exists a separately continuous mapping $f:X^2\to Z$ with the diagonal $g$.
\end{theorem}

\begin{proof}
$(1)$. The implication $(i)\Rightarrow (ii)$ follows from Corollary~\ref{cor:3.2}, and the implication $(ii)\Rightarrow (i)$ follows from the inclusion $CC(X\times X,Z)\subseteq B_1(X^2,Z)$.

$(2)$ and $(3)$. In the case when  $X$ is a metrically quarter-stratifiable strongly countably-dimensional paracompact space or $X$ is a metrically quarter-stratifiable space and $Z$ is a locally convex equiconnected space (see definitions in~\cite{B}), the inclusion \mbox{$CC(X\times X,Z)\subseteq B_1(X^2,Z)$} follows from \cite[Corollary~5.2]{B}.

$(4)$. Notice that an equiconnected space $Z$ is locally arcwise connected. Therefore, by \cite{Ka}, the inclusion $CC(X\times X,Z)\subseteq B_1(X^2,Z)$ holds.

\end{proof}

The conditions $(1)$ -- $(4)$ of Theorem \ref{th:3.1} suggest the following question.

\begin{question}\label{qu:3.7}
  Let $X$ be a quarter-stratifiable space and let $Z$ be a metrizable path-connected space. Does the inclusion $CC(X\times X,Z)\subseteq B_1(X^2,Z)$ hold?
\end{question}

Let us remark that in the previous theorem the Baire classification of diagonals follows from the stronger inclusion $CC\subseteq B_1$. Hence, naturally the following question appears.
\begin{question}\label{qu:3.6}
  Let $X$ and $Z$ be topological spaces such that {the} diagonal of every mapping  $f\in CC(X\times X,Z)$ is of the first Baire class. Does the inclusion $CC(X\times X,Z)\subseteq B_1(X^2,Z)$ hold?
\end{question}

\section{Necessary conditions on diagonals of $CL$-mappings}

Let  $X$ and $Z$ be metric spaces.  A mapping $f:X\to Z$ is called {\it Lipschitz} if the Lipschitz constant $${\rm Lip}(f)=\sup\limits_{x,y\in X, x\ne y}\frac{|f(x)-f(y)|_Z}{|x-y|_X}$$ is finite. We say that {\it $f$ is Lipschitz with a constant $C$} if $C\geq {\rm Lip}(f)$. A mapping $f:X\to Z$ is {\it Lipschitz on a set $A\subseteq X$ (with a constant $C\geq 0$)} if $f|_A$ is Lipschitz (with a constant $C$). The set of all Lipschitz mappings $f:X\to Z$ we denote by  $L(X,Z)$.

A mapping $f:X\to Z$ is called  {\it Lipschitz at a point $x\in X$} if the pointwise Lipschitz constant $${\rm Lip}_x(f)=\sup\limits_{y\in X, y\ne x}\frac{|f(x)-f(y)|_Z}{|x-y|_X}$$ is finite. We say that {\it $f$ is pointsise Lipschitz} if $f$ is Lipschitz at every point $x\in X$.

\begin{proposition}\label{pr:4.2}
  Let $X$ be a topological space, $Y, Z$ metric spaces, $g:X\to Y$ a continuous mapping, $C,\delta >0$ and $f:X\times Y\to Z$ a mapping which is continuous with respect to the {first variable} and such that for any $x\in X$ and $y\in Y$ with $|y-g(x)|_Y<\delta$ the inequality $$|f(x,y)-f(x,g(x))|_Z\leq C|y-g(x)|_Y$$ holds. Then the mapping $h:X\to Z$, $h(x)=f(x,g(x))$, is continuous.
\end{proposition}

\begin{proof}
 Fix $x_0\in X$ and $\varepsilon>0$. Denote $y_0=g(x_0)$ and $z_0=f(x_0,y_0)=h(x_0)$. Using the continuity of {the} mappings $g$ and $f_{y_0}$ at the point  $x_0$ we choose a neighborhood $U$ of $x_0$ such that $|g(x)-y_0|_Y<{\rm min}\{\frac{\varepsilon}{2C}, \delta\}$ and $|f(x,y_0)-z_0|_Z<\frac{\varepsilon}{2}$ for every $x\in U$. Now for every $x\in U$, using the hypothesis of the proposition, we conclude that  $$|h(x)-z_0)|_Z=|f(x,g(x))-z_0|_Z\leq |f(x,g(x))-f(x,y_0)|_Z+|f(x,y_0)-z_0|_Z<$$
$$< C|g(x)-y_0|_Y+\frac{\varepsilon}{2}<\frac{\varepsilon}{2}+\frac{\varepsilon}{2}=\varepsilon.$$
Hence, $h$ is continuous at the point $x_0$.
\end{proof}

Let  $X$ and $Y$ be topological spaces. A mapping $f:X\to Y$ is said to be {\it $\sigma$-continuous} (see \cite{BKMM}), if there exists a sequence  $(A_n)_{n=1}^{\infty}$ of closed sets $A_n\subseteq X$ such that all the restrictions  $f|_{A_n}$ are continuous. In \cite{BaBo} such functions are called {\it piecewise continuous}.

The following result for $CL$-mappings on a product of a Baire and a metric spaces was, in fact, established in~\cite[Theorem~2]{HMM}.

\begin{theorem}\label{th:4.3}
   Let $X$ be a topological space, $Y, Z$ metric spaces, $g:X\to Y$ a continuous mapping and $f:X\times Y\to Z$ a mapping which is continuous with respect to {the first variable} and pointwise Lipschitz with respect to {the second variable} at every point of the set  $E=\{(x,g(x)):x\in X\}$. The the mapping $h:X\to Z$, $h(x)=f(x,g(x))$, is $\sigma$-continuous.
\end{theorem}

\begin{proof}
For every $n\in\mathbb N$ we denote by $A_n$ the set of all points $x\in X$ such that for every $y\in Y$ with $|y-g(x)|_Y<\frac{1}{n}$ the inequality $$|f(x,y)-f(x,g(x))|_Z\leq n|y-g(x)|_Y$$ holds. We show that $\overline{A_n}\subseteq A_{2n}$ for every $n\in\mathbb N$.

To obtain a contradiction, suppose that  $x_0\not\in A_{2n}$. We show that $x_0\not\in \overline{A_{n}}$. We choose $y\in Y$ such that $|y-y_0|_Y<\frac{1}{2n}$ and $|f(x_0,y)-f(x_0,y_0)|_Z>2n|y-y_0|_Y$, where $y_0=g(x_0)$. The continuity of the mapping $\varphi(x)=\frac{|f(x,y)-f(x,y_0)|_Z}{|y-g(x)|_Y}$ at $x_0$ implies the existence of a neighborhood  $U_1$ of  $x_0$ such that $\varphi(x)>2n$ for every $x\in U_1$. Notice that $|y-y_0|_Y>0$.
Therefore, the continuity of the mapping $\psi(x)=|g(x)-y_0|_Y-|y-g(x)|_Y$ at $x_0$ and the condition  $\psi(x_0)<0$ imply the existence of a neighborhood $U_2$ of  $x_0$ such that $|g(x)-y_0|_Y< |y-g(x)|_Y$ and $|g(x)-y_0|_Y<\frac{1}{2n}$ for every $x\in U_2$. We set $U_0=U_1\cap U_2$ and prove that $U_0\cap A_n=\emptyset$.

 Let $x\in U_0$. Since $|f(x,y)-f(x,y_0)|_Z\leq |f(x,y)-f(x,g(x))|_Z+|f(x,g(x))-f(x,y_0)|_Z$, there exists $\tilde{y}\in\{y, y_0\}$ such that $|f(x,\tilde{y})-f(x,g(x))|_Z\geq \frac{1}{2}|f(x,y)-f(x,y_0)|_Z$. Remark that $|y-g(x)|_Y\leq |y-y_0|_Y+ |y_0-g(x)|_Y<\frac{1}{n}$ and $|y_0-g(x)|_Y<\frac{1}{2n}<\frac{1}{n}$. Hence, $|\tilde{y}-g(x)|_Y<\frac{1}{n}$ and $$\frac{|f(x,\tilde{y})-f(x,g(x))|_Z}{|\tilde{y}-g(x)|_Y}\geq \frac{1}{2}\frac{|f(x,y)-f(x,y_0)|_Z}{|y-g(x)|_Y}=\frac{1}{2}\varphi(x)>n.$$
Therefore, $U_0\cap A_n=\emptyset$ and $\overline{A_n}\subseteq A_{2n}$.

For every $n\in\mathbb N$ let $F_n=\overline{A_n}$. Since $F_n\subseteq A_{2n}$, Proposition~\ref{pr:4.2} implies the continuity of $h|_{F_n}$. It remains to observe that  $\bigcup\limits_{n=1}^{\infty}F_n=X$.
\end{proof}

\begin{corollary}\label{cor:5.1}
 Let $X, Z$ be metric spaces and let $f:X^2\to Z$ be a mapping which is continuous with respect to the first variable and pointwise Lipschitz with respect to the second one at every point of the diagonal  $\Delta=\{(x,x):x\in X\}$. Then the mapping $g:X\to Z$, $g(x)=f(x,x)$, is $\sigma$-continuous.
\end{corollary}

\section{Characterization of diagonals of $CL$-mappings}

We first prove several auxiliary facts.

\begin{lemma}\label{lm:4.6}
  Let $X$ be a metric space and $G$ an open set in $X^2$ such that $\Delta=\{(x,x):x\in X\}\subseteq G$. Then there exists a continuous function  $g:X\to (0,+\infty)$ such that  $\{(x,y)\in X^2:|x-y|_X<g(x)\} \subseteq G$.
  \end{lemma}

\begin{proof}
For every point $x\in X$ we choose a neighborhood $U_x$ of $x$ in $X$ and a number $\varepsilon_x>0$ such that for any $y\in U_x$ the inclusion $$\{(x',y')\in X^2:|x'-y|_X<\varepsilon_x, |y'-y|_X<\varepsilon_x\}\subseteq G$$ holds. Let $(V_i:i\in I)$ be a locally finite open cover of $X$ which refines  $(U_x:x\in X)$, and $(\varphi_i:i\in I)$ a locally finite partition of unity on  $X$ subordinated to $(V_i:i\in I)$. For every $i\in I$ we choose $x\in X$ such that $V_i\subseteq U_x$ and let $\delta_i=\varepsilon_x$. It is easy to verify that the function $g:X\to (0,+\infty)$, $g(x)=\sum\limits_{i\in I}\delta_i \varphi_i(x)$, is to be found.
\end{proof}

\begin{lemma}\label{lm:4.7}
  Let $X$ be a metric space, $(H_n)_{n=1}^{\infty}$ a sequence of open in $X^2$ sets $H_n$ such that $\Delta=\{(x,x):x\in X\}\subseteq H_{n}$ for every $n\in\mathbb N$. Then there exist a sequences $(G_n)_{n=1}^{\infty}$ and $(F_n)_{n=1}^{\infty}$ of open sets $G_n$ and closed sets $F_n$ in $X^2$ and sequences $(\varphi_n)_{n=1}^{\infty}$ and $(\delta_n)_{n=1}^{\infty}$ of separately Lipschitz functions $\varphi_n:X^2\to [0,1]$ and positive functions  $\delta_n:X\to (0,+\infty)$ with the following conditions:

  $(1)$\,\,\, $\Delta\subseteq G_{n+1}\subseteq F_n\subseteq G_n\subseteq H_n$ for every $n\in\mathbb N$;

  $(2)$\,\,\, $X^2\setminus G_n\subseteq\varphi_n^{-1}(0)$ and $F_n\subseteq \varphi_n^{-1}(1)$ for every $n\in\mathbb N$;

  $(3)$\,\,\, $|u-v|_X\geq \delta_n(x)$ for every $x\in X$, $n\in\mathbb N$ and $u,v\in X$ with $(x,u)\in G_{n+1}$ and $(x,v)\in X^2\setminus F_n$.
  \end{lemma}

\begin{proof} Using Lemma~\ref{lm:4.6} we construct a sequence $(g_n)_{n=1}^{\infty}$ of continuous functions $g_n:X\to (0,+\infty)$ such that $g_{n+1}(x)<\frac{1}{16}g_n(x)$ for all $n\in\mathbb N$ and $x\in X$, and $G_n=\{(x,y)\in X^2:|x-y|_X<g_n(x)\} \subseteq H_n$ for every $n\in\mathbb N$. Clearly, the sequences $(G_n)_{n=1}^{\infty}$ and $(F_n)_{n=1}^{\infty}$ of sets $F_n=\{(x,y)\in X^2:|x-y|_X\leq\frac{1}{8}g_n(x)\}$ satisfy condition~$(1)$. Moreover the sequence $(\delta_n)_{n=1}^{\infty}$ of function $\delta_n(x)=\frac{1}{16}g_n(x)$ satisfies condition~$(3)$.

Denote by  $|\cdot-\cdot|_{Z}$ the metric on $Z=X^2$ defined by the formula: $|z_1-z_2|_Z=\max\{|x_1-x_2|_X, |y_1-y_2|_X\}$, where $z_1=(x_1,y_1)$ and $z_2=(x_2,y_2)$. For any $n\in\mathbb N$ and $z\in Z$ let $\alpha(z)=|z-F_n|_Z$, $\beta(z)=|z-Z\setminus G_{n}|_Z$ and $\varphi_n(z)=\frac{\beta(z)}{\alpha(z)+\beta(z)}$. Evidently, all the functions $\varphi_n$ satisfy~$(2)$. It remains to show that all the functions $\varphi_n$ are separately Lipschitz.

Fix $n\in\mathbb N$ and $x_0\in X$. Using the continuity of $g_n$ at $x_0$, we choose a positive number $\varepsilon\leq \frac{g_n(x_0)}{8}$ such that $g_n(x)>\frac{7g_n(x_0)}{8}$  and $g_{n}(x)< 2 g_{n}(x_0)$ whenever $|x-x_0|_X<\varepsilon$. Prove that $\max\{\alpha(z),\beta(z)\}\geq \varepsilon$ for every $z=(x_0,y)\in Z$. Suppose that there exist $y\in X$, $z_1=(x_1,y_1)\in F_n$ and $z_2=(x_2,y_2)\in Z\setminus G_n$ such that $|z-z_1|_Z<\varepsilon$ and $|z-z_2|_Z<\varepsilon$, where $z=(x_0,y)$. Since $|x_0-x_1|_X<\varepsilon$ and $(x_1,y_1)\in F_n$, $|x_1-y_1|_X\leq \frac{1}{8}g_{n}(x_1)<\frac{1}{4}g_n(x_0)$. Analogously, since $|x_0-x_2|_X<\varepsilon$ and $(x_2,y_2)\in Z\setminus G_n$, $|x_2-y_2|_X\geq g_{n}(x_2)>\frac{7g_n(x_0)}{8}$. Now
$$|y_2-y_1|_X\geq |y_2-x_2|_X - |x_2-x_0|_X - |x_0-x_1|_X - |x_1-y_1|_X\geq\frac{3}{8}g_n(x_0).$$
But
$$|y_2-y_1|_X\leq |y_2-y|_X +|y-y_1|_X\leq\frac{1}{4}g_n(x_0),$$
a contradiction.

Hence, $\max\{\alpha(z),\beta(z)\}\geq \varepsilon$ for every $(x_0,y)\in Z$. Since the functions $\alpha$, $\beta$  and the function $\varphi:(0,+\infty)^2\setminus [0,\varepsilon]^2\to \mathbb R$, $\varphi(s,t)=\frac{t}{s+t}$, are Lipschitz and $\varphi_n(z)=\varphi(\alpha(z),\beta(z))$ for every $z=(x_0,y)$, the function $\varphi_n$ is Lipschitz with respect to the first variable.

Similarly one can prove that $\varphi_n$ is Lipschitz with respect to the second variable.
\end{proof}

We say that {\it $f:X\to Z$ is bounded} if there exists $C\geq 0$ such that \mbox{$|f(x)-f(y)|_Z\leq C$} for every $x,y\in X$.

\begin{proposition}\label{pr:4.0}
  Let  $X$ and $Z$ be metric spaces, let $f:X\to Z$ be a bounded mapping and let $A,B\subseteq X$ and $\delta>0$ be such that $X=A\cup B$, $f$ is Lipschitz on $A$ and on $B$ and $|x-y|_X\geq \delta$ for every $x\in X\setminus A$ and $y\in X\setminus B$. Then $f$ is Lipschitz.
\end{proposition}

\begin{proof}
Choose $C\geq 0$ such that $|f(y)-f(x)|_Z\leq C$ for every $x,y \in X$. It is easy to see that ${\rm Lip}(f)\leq \max\{{\rm Lip}(f|_A),{\rm Lip}(f|_B),\frac{C}{\delta}\}$.
 \end{proof}

A topological space $Z$ is called {\it an absolute extensor for a topological space $X$}, if for any closed set $A\subseteq X$ every continuous mapping  $g:A\to Z$ can be extended to a continuous mapping $f:X\to Z$.

Let $X$ and $Z$ be topological spaces. A mapping $f:X\to Z$ is said to be {\it a mapping of the stable first Baire class}, if there exists a sequence  $(f_n)^{\infty}_{n=1}$ of continuous mappings $f_n:X\to Z$ which pointwise stably converges to  $f$, i.e. for every $x\in X$ the sequence $(f_n(x))_{n=1}^{\infty}$ stably converges to $f(x)$. The collection of all mappings  $f:X\to Z$ of the stable first Baire class we denote by $B^d_1(X,Z)$.

\begin{theorem}\label{th:4.5}
  Let $X$ be a metric space, $(Z,\lambda)$ a metric equiconnected space, where $\lambda$ is Lipschitz with respect to third variable and $g\in B^d_1(X,Z)$. Then there exists a mapping $f:X^2\to Z$ with the diagonal $g$ which is continuous with respect to the first variable and Lipschitz with respect to the second one.
\end{theorem}

\begin{proof}
Let $(g_n)^{\infty}_{n=1}$ be a sequence of continuous mappings $g_n:X\to Z$ which pointwise stably converges to $g$. Let $G_0=F_0=X^2$ and $$H_n=\{(x,y)\in X^2: |g_n(x)-g_n(y)|_Z<\frac1n,\,|g_{n+1}(x)-g_{n+1}(y)|_Z<\frac1n\}$$  for every $n\in\mathbb N$.
By Lemma~\ref{lm:4.7} there exist sequences $(G_n)_{n=1}^{\infty}$ and $(F_n)_{n=1}^{\infty}$ of open sets $G_n$ and closed sets $F_n$ in $X^2$ and sequences $(\varphi_n)_{n=1}^{\infty}$ and $(\delta_n)_{n=1}^{\infty}$ of separately Lipschitz functions $\varphi_n:X^2\to [0,1]$ and positive functions  $\delta_n:X\to (0,+\infty)$ which satisfy conditions $(1)$, $(2)$ and $(3)$ of Lemma~\ref{lm:4.7}. According to Theorem~\ref{th:3.1} the mapping $f:X^2\to Z$ defined by (\ref{eq:3.0}) is separately continuous. Clearly that $g$ the diagonal of $f$. It remains to verify that $f$ is Lipschitz with respect to the second variable.

Fix $x_0\in X$ and show that the function $h=f^{x_0}:X\to Z$, $h(y)=f(x_0,y)$, is Lipschitz. It follows from (\ref{eq:3.1}) that
\begin{equation}\label{eq:5.0}
h(y)=\lambda(g_n(x_0),g_{n+1}(x_0),\varphi_n(x_0,y))
\end{equation}
for every $n\in\mathbb N$ and $y\in A_n=\{y\in X: (x_0,y)\in F_{n-1}\setminus G_{n+1}\}$. Since for every $n\in\mathbb N$ the set $h(A_n)$ contained in the compact set $\{\lambda(g_n(x_0),g_{n+1}(x_0),t):t\in[0,1]\}$, the mapping $h|_{A_n}$ is bounded. {Using $(i)$, we} choose a number $n_0\in\mathbb N$ such that $g_n(x_0)=g(x_0)$ for all $n>n_0$. Then $h(X)=\bigcup\limits_{n=1}^{n_0}h(A_n)$ and $h$ is bounded.

Since $\lambda$ is Lipschitz with respect to the third variable and $\varphi_n$ is Lipschitz with respect to the second variable, it follows from (\ref{eq:5.0}) that $h$ is is Lipschitz on each set $A_n$. For every $n\in\mathbb N$ put $B_n=\bigcup\limits_{k=1}^{n}A_k$. Note that For every $n\in\mathbb N$ if $u\in A_{n+1}\setminus B_n$ and $v\in B_n\setminus A_{n+1}$ then $(x_0,u)\in G_{n+1}$, $(x_0,v)\in X^2\setminus F_n$ and $|u-v|_X\geq \delta_n(x_0)$ according to $(3)$ of Lemma~\ref{lm:4.7}. Now it easy to verify by induction on $n$ that according to Proposition \ref{pr:4.0} $h$ is Lipschitz on each set $B_n$. Moreover $h(y)=g(x_0)$ for every $y\in A=\{u\in X:(x_0,u)\in F_{n_0}\}$, in particular, $h$ is Lipschitz on $A$. Therefore by Proposition \ref{pr:4.0}, $h$ is Lipschitz on $X=A\cup B_{n_0}$.
\end{proof}

\begin{theorem}\label{th:4.6}
  Let $X$ be a metric space, $(Z,\lambda)$ a metric equiconnected space, where $\lambda$ is Lipschitz with respect to third variable and $Z$ is an absolute extensor for $X$, and let $g:X\to Z$ be a mapping. Then the following conditions are equivalent:

 $(i)$\,\,\,$g\in B^d_1(X,Z)$;

 $(ii)$\,\,\,there exists a mapping $f:X^2\to Z$ with the diagonal $g$ which is continuous with respect to the first variable and Lipschitz with respect to the second one;

 $(iii)$\,\,\,there exists a mapping $f:X^2\to Z$ with the diagonal $g$ which is continuous with respect to the first variable and pointwise Lipschitz with respect to the second variable at every point of the diagonal $\Delta=\{(x,x):x\in X\}$.
\end{theorem}

\begin{proof}
The implication $(i)\Rightarrow (ii)$ follows from Theorem \ref{th:4.5}.

The implication $(ii)\Rightarrow (iii)$ is obvious.

$(iii)\Rightarrow (i)$. By Corollary~\ref{cor:5.1}, $g$ is $\sigma$-continuous. Now according to \cite[Theorem 6.3]{BaBo}, $g\in B^d_1(X,Z)$.
\end{proof}

The following questions concern the essentiality of come conditions in Theorems~\ref{th:4.6}.

\begin{question}\label{qu:4.9}
Let $X$ be a metric space, $(Z,\lambda)$ a metric equiconnected space and $f:X^2\to Z$ a mapping which is continuous with respect to the first variable and Lipschitz with respect to the second one. Does the mapping $g:X\to Z$, $g(x)=f(x,x)$, belong to the first stable Baire class?
\end{question}

\section{Characterization of diagonals of $CD$-mappings}

 Let $X$ and  $Y$ be normed spaces. A mapping $f:X\to Y$ is {\it Fr\'{e}chet differentiable at a point  $x_0\in X$} if there exists a linear continuous operator $T:X\to Y$ such that $$\displaystyle{\lim\limits_{t\to 0}\frac{f(x_0+tx)-f(x_0) - T(x)}{t}}=0.$$ If  $f:X\to Y$ is Fr\'{e}chet differentiable at every point $x\in X$, then $f$ is called {\it Fr\'{e}chet differentiable}.

\begin{lemma}\label{lm:7.1}
  Let $X$ be a normed space with (continuously) Fr\'{e}chet differentiable norm, $(H_n)_{n=1}^{\infty}$ a sequence of open in $X^2$ sets $H_n$ such that $\Delta=\{(x,x):x\in X\}\subseteq H_{n}$ for every $n\in\mathbb N$. Then there exist a sequence $(G_n)_{n=1}^{\infty}$ of open in $X^2$ sets $G_n$ and a sequence $(\varphi_n)_{n=1}^{\infty}$ of jointly continuous, Lipschitz  and (continuously) differentiable with respect to the second variable functions $\varphi_n:X^2\to [0,1]$ with the following conditions:

  $(1)$\,\,\, $\Delta\subseteq G_{n+1}\subseteq \overline{G_{n+1}}\subseteq G_n\subseteq H_n$ for every $n\in\mathbb N$;

  $(2)$\,\,\, $X^2\setminus G_n\subseteq\varphi_n^{-1}(0)$ and $\overline{G_{n+1}}\subseteq \varphi_n^{-1}(1)$ for every $n\in\mathbb N$.
  \end{lemma}

\begin{proof}
We first  using Lemma \ref{lm:4.6} construct a sequence $(g_n)_{n=1}^{\infty}$ of continuous functions $g_n:X\to (0,+\infty)$ such that $g_{n+1}(x)<g_n(x)$ for all $n\in\mathbb N$ and $x\in X$, and $G_n=\{(x,y)\in X^2:\|x-y\|<g_n(x)\} \subseteq H_n$ for every $n\in\mathbb N$, and choose a continuously differentiable function $\psi:\mathbb R\to[0,1]$ such that $\psi(x)=1$ if $x\leq 0$ and $\psi(x)=0$ if $x\geq 1$. It remains to observe that $\varphi_n(x,y)=\psi\left(\frac{\|x-y\|-g_{n+1}(x)}{g_n(x)-g_{n+1}(x)}\right)$ for all $n\in\mathbb N$ and $x,y\in X$.
\end{proof}

\begin{theorem}\label{th:7.2}
   Let $X$ be a normed space with a (continuously) Fr\'{e}chet differentiable norm, $Z$ a normed space and $g:X\to Z$. Then the following conditions are equivalent:

 $(i)$\,\,\,$g\in B^d_1(X,Z)$;

 $(ii)$\,\,\,there exists a mapping $f:X^2\to Z$ with the diagonal  $g$ which is continuous with respect to the first variable and is Lipschitz and continuously Fr\'{e}chet differentiable with respect to the second variable;

 $(iii)$\,\,\,there exists a mapping $f:X^2\to Z$ with the diagonal $g$ which is continuous with respect to the first variable and pointwise Lipschitz with respect to the second variable at every point of the diagonal $\Delta=\{(x,x):x\in X\}$.
\end{theorem}

\begin{proof}
$(i)\Rightarrow (ii)$. Let $(g_n)^{\infty}_{n=1}$ be a sequence of continuous mappings $g_n:X\to Z$ which pointwise stably converges to $g$. Let $G_0=F_0=X^2$ and $$H_n=\{(x,y)\in X^2: |g_n(x)-g_n(y)|_Z<\frac1n,\,|g_{n+1}(x)-g_{n+1}(y)|_Z<\frac1n\}$$  for every $n\in\mathbb N$.
By Lemma~\ref{lm:7.1} there exist  a sequence $(G_n)_{n=1}^{\infty}$ of open on $X^2$ sets $G_n$ and a sequence $(\varphi_n)_{n=1}^{\infty}$ of jointly continuous, Lipschitz  and (continuously) differentiable with respect to the second variable functions $\varphi_n:X^2\to [0,1]$ which satisfy conditions $(1)$ and $(2)$ of Lemma~\ref{lm:7.1}. For  every $n\in\mathbb N$ let $F_n=\overline{G_{n+1}}$. According to Theorem~\ref{th:3.1} the mapping $f:X^2\to Z$ defined by (\ref{eq:3.0}), where $\lambda(x,y,t)=(1-t)x+ty$, is separately continuous.

Remark that $\varphi_n(x,y)\varphi_{n+1}(x,y)=\varphi_{n+1}(x,y)$ for every $n\in\mathbb N$, therefore for {the} normed space  $Y$ condition~(\ref{eq:3.1}) is of the following form:
\begin{gather*}
f(x,y)=(1-\varphi_n(x,y))g_n(x)+\varphi_n(x,y)g_{n+1}(x)-
\end{gather*}
\begin{gather}\label{eq:4.2}
-\varphi_{n+1}(x,y)g_{n+1}(x)+\varphi_{n+1}(x,y)g_{n+2}(x)
\end{gather}
for all $(x,y)\in F_{n-1}\setminus F_{n+1}$.

Now it follows from the properties of functions $\varphi_n$ that $f$ is (continuously) Fr\'{e}chet differentiable and Lipschitz with respect to the second variable on the open set $G_0\setminus F_1=F_0\setminus F_1$ and on every open set $G_{n}\setminus F_{n+1}$,  $n\in\mathbb N$. Hence, $f$ is (continuously) Fr\'{e}chet differentiable and Lipschitz with respect to the second variable on the open set $X^2\setminus E$, where $E=\bigcap\limits_{n=1}^{\infty}G_n$ (the and Lipschitz property can be obtained by Proposition \ref{pr:4.0}).

Fix $(x_0,y_0)\in E$ and choose a number $n_0\in\mathbb N$ such that  $g_n(x_0)=g(x_0)$ for all $n>n_0$. Then $f(x_0,y)=g(x_0)$ for al $y\in \{x\in X: (x_0,x)\in G_{n_0+1}\}$. Therefore, $f$ is (continuously) Fr\'{e}chet differentiable with respect to the second variable at $(x_0,y_0)$. Moreover $f^{x_0}$ is Lipschitz by Proposition \ref{pr:4.0}.

The implication $(ii)\Rightarrow (iii)$ is obvious.

$(iii)\Rightarrow (i)$. By Corollary~\ref{cor:5.1}, $g$ is $\sigma$-continuous. Since by Dugundji theorem normed space $Z$ is an absolute extensor for $X$,  $g\in B^d_1(X,Z)$ by \cite[Theorem 6.3]{BaBo}.
\end{proof}

\begin{question}\label{qu:4.10}
 Let $X$, $Z$ be normed spaces and  $g\in B^d_1(X,Z)$. Does there exist a mapping $f:X^2\to Z$ with the diagonal $g$ which is continuous with respect to the first variable and is Lipschitz and continuously Fr\'{e}chet differentiable with respect to the second one?
\end{question}

\begin{question}\label{qu:4.11}
Let $X, Z$ be normed spaces and $f:X^2\to Z$ a mapping which is continuous with respect to the first variable and (weakly) G$\hat{a}$teaux differentiable with respect to the second one. Does the mapping $g:X\to Z$, $g(x)=f(x,x)$, belong to the first stable Baire class?
\end{question}


\end{document}